\documentclass[11pt, letter]{amsart}

\usepackage[latin1]{inputenc}
\usepackage[T1]{fontenc}
\usepackage{amssymb}
\usepackage{amsmath}
\usepackage{amsthm}
\usepackage[english]{babel}
\usepackage{graphicx}
\usepackage{epsfig}
\usepackage[mathscr]{eucal}
\usepackage{color}
\usepackage{latexsym}
\usepackage{hyperref}
\usepackage{stmaryrd}
\usepackage[all]{xy}
\usepackage{url}
\usepackage{ifthen}
\usepackage[obeyFinal]{todonotes}

\usepackage{tikz}
\usetikzlibrary{matrix, arrows}

\theoremstyle{plain}
\newtheorem{thm}{Theorem}[section]
\newtheorem{theorem}[thm]{Theorem}

\newtheorem{proposition}[thm]{Proposition}
\newtheorem{corollary}[thm]{Corollary}

\newtheorem{lemma}[thm]{Lemma}
\newtheorem{maintheorem}{Theorem}

\theoremstyle{definition}
\newtheorem{definition}[thm]{Definition}
\newtheorem{question}[thm]{Question}
\theoremstyle{remark}
\newtheorem{ex}[thm]{Example}
\newtheorem{non-ex}[thm]{Counterexample}
\newtheorem*{remark}{Remark}
\theoremstyle{plain}

\newcommand{\llp}{\boxslash}
\AtBeginDocument{\DeclareMathSymbol{:}{\mathpunct}{operators}{"3A}}
\newcommand{\LL}{\mathcal{L}}
\newcommand{\RR}{\mathcal{R}}
\newcommand{\C}{\mathcal{C}}
\newcommand{\D}{\mathcal{D}}

\newcommand{\CC}{\mathbb{C}}
\newcommand{\DD}{\mathbb{D}}
\newcommand{\initial}{\varnothing}
\newcommand{\terminal}{{*}}

\renewcommand{\P}{\mathscr{P}}

\newcommand{\W}{\mathcal{W}}
\DeclareMathOperator{\Hom}{Hom}
\DeclareMathOperator{\ob}{ob}

\DeclareMathOperator{\iso}{iso}
\DeclareMathOperator{\Ho}{Ho}
\DeclareMathOperator*{\colim}{colim}

\newcommand{\rto}{\to}

\newcommand{\fib}{\ar@{->>}}
\newdir{ (}{{}*!/-5pt/\dir^{(}}
\newcommand{\cofib}{\ar@{ (->}} 
\newcommand{\cofibfib}{\ar@{ (->>}} 

\newcommand{\acycfib}{\stackrel{\sim}{\twoheadrightarrow}}
\renewcommand{\tilde}{\widetilde}

\renewcommand{\emph}{\textsl}

\date{\today}

\begin{document}

\title[Extending model structures]{Extending to a model structure is not a first-order property}
\author[Droz Zakharevich]{Jean-Marie Droz \and Inna Zakharevich}
\begin{abstract}
  Let $\C$ be a finitely bicomplete category and $\W$ a subcategory.  We prove
  that the existence of a model structure on $\C$ with $\W$ as the subcategory
  of weak equivalence is not first order expressible.  Along the way we
  characterize all model structures where $\C$ is a partial order and show that
  these are determined by the homotopy categories.
\end{abstract}

\maketitle


\section*{Introduction}

What is a ``homotopy theory''?  Colloquially, it is a context in which one
classifies objects up to ``weak equivalence'' instead of up to isomorphism:
chain complexes up to quasiisomorphism \cite[Section 4 Remarks]{quillen67},
topological spaces up to homotopy equivalence \cite{strom} or weak equivalence
\cite[Theorem II.3.1]{quillen67}, categories up to functors which are homotopy
equivalences on geometric realization \cite{thomason}, etc.  These can be
modeled and formalized in many different ways.  However, if we wish to show that
homotopy theories are equivalent (possibly up to ``homotopy'', inside a
``homotopy theory of homotopy theories'') then we often use Quillen's model
categories \cite{quillen67} to prove this.  A model category has three
distinguished classes of morphisms---weak equivalences, cofibrations and
fibrations---of which only the weak equivalences characterize the homotopy
theory.  The cofibrations and fibrations are there purely to assist in
calculations and constructions.  Due to the presense of this extra structure,
model categories are rigid and have many computational and formal methods for
working with them.  However---and also due to the presense of extra
structure---they generally do not appear fully formed: they often arise in
situations where the weak equivalences are known, but the choice of cofibrations
and fibrations is not.

More concretely, we know that model categories naturally produce models of
homotopy theories, such as quasi-categories \cite{joyal, lurie}, simplicially
enriched categories \cite{quillen67, dwyerkan1, dwyerkan2} and complete Segal
spaces \cite{bergner07}.  However, there is no known way of identifying which
quasi-categories (for instance) arise from model structures.\footnote{For
  \textsl{combinatorial} model categories there is: if the $\infty$-category
  represented by the pair is locally presentable. \cite[Proposition
  A.3.7.6]{lurie}} Thus model categories live
in a strange gray area of homotopy theory: we know that all models of the
homotopy theory of homotopy theories form model categories, and we know that
they are all equivalent as model categories.  However, we do not know which
parts of the homotopy theory of homotopy theories can be explored purely using
the theory of model categories.

One easy place to start the comparison would be with Barwick--Kan's model of
relative categories \cite{barwickkan, DHKS}.  A relative category is simply a
pair $(\C,\W)$ of a category and a subcategory of weak equivalences.  It is
known that the category of relative categories is a model category, which is
Quillen equivalent to the other models of the homotopy theory of homotopy
theories.  To try to identify which homotopy theories arise from model
categories is to answer the question of when a pair $(\C,\W)$ of a category and
a subcategory of weak equivalences extends to a model structure.  In a few
cases, specialized techniques can be used to construct model structures, for
example cofibrant generation \cite[Section 2.1]{hovey99}, Bousfield localization
\cite[Chapter 4]{hirschhorn},  Cisinski's minimal model structures
\cite{cisinski06} or one-dimensional model structures studied in
\cite{rosickytholen, balchingarner}.  However, there is no practically useful necessary and
sufficient criterion for determining whether it is possible to complete a pair
$(\C,\W)$ to a model structure.

In this paper, we show that in a well defined sense such a criterion does not
exist, and therefore that there is no simple way to characterize which homotopy
theories can be accessed by model categories.  More concretely, we show that
there is no set of first order formulae (i.e. formulae quantifying only over
elements but not over sets or functions) which can identify those relative
categories that can extend to model categories.  Note that if we allow
quantifying over sets (or proper classes), we can identify model structures by
simply stating that there exist sets of cofibrations and fibrations that satisfy
the model category axioms.  One could hope, however, that there exist simpler
formulas, which only quantified over objects or morphisms in categories, that
could identify which relative categories arise from model categories.  This is
unfortunately not the case.

\begin{maintheorem} \label{thm:second-order}
  In the language of categories with a designated subcategory of weak
  equivalence there is no first order characterization of those that extend to
  model structures.
\end{maintheorem}

To prove this theorem we construct two different pairs $(\C,\W)$ and $(\C',\W')$
that satisfy all of the same first-order statements but such that $(\C,\W)$ does
\emph{not} extend to a model structure while $(\C',\W')$ does.  To accomplish
this we produce a complete characterization of all model categories whose
underlying categories are \emph{posets}: those skeletal categories for which
$|\Hom(A,B)|\leq 1$ for all $A$ and $B$.  We then show that a simpler
characterization exists when the underlying category is countable, and use this
to produce the desired pairs.

The characterization of model categories on posets is the following.

\begin{maintheorem} \label{thm:modelexistsall} Let $\C$ be a preorder closed
  under finite limits and colimits,
  and let $\W$ be a subcategory of $\C$.  A model structure exists on $\C$ with
  weak equivalences $\W$ if and only if the following conditions hold.
  \begin{itemize}
  \item[(a)] For any two composable morphisms $f$ and $g$ in $\C$, if $gf$ is in $\W$
    then $f$ and $g$ are in $\W$.
  \item[(b)] There exists a functor $\chi: \C \rto \C$ such that $\chi(\W) \subseteq
    \iso \C$, and for every object $A \in \C$, the diagram
    \[\xymatrix{ A\times \chi(A) \ar[r] \ar[d] & \chi(A) \ar[d] \\ A \ar[r] & A
      \cup \chi(A)}\]
    lies in $\W$.
  \end{itemize}
\end{maintheorem}

When $\C$ is countable we prove a stronger statement:
\begin{maintheorem} \label{thm:modelexistscount}
  If $\W$ has only a countable number of connected components then there is a
  first-order characterization of when $(\C,\W)$ extends to a model structure.
\end{maintheorem}
Thus when $\C$ is countable the existence of a model structure extending
$(\C,\W)$ is a first-order condition. Therefore in order to construct the
desired counterexample it suffices to construct a pair $(\C,\W)$ that satisfies
the first order characterization from Theorem~\ref{thm:modelexistscount} (but
not the condition of countability!) and does not satisfy the conditions of
Theorem~\ref{thm:modelexistsall}. By the L\"owenheim--Skolem \cite{maltsev}
Theorem there exists a countable model $(\C',\W')$ which satisfies all of the
same first-order statements that $(\C,\W)$ does.  By
Theorem~\ref{thm:modelexistscount} the pair $(\C',\W')$ extends to a model
structure while $(\C,\W)$ does not.

As an interesting aside, we also prove in Theorem~\ref{thm:C->HoC} that any two
model structures on a poset that have the same weak equivalences are Quillen
equivalent.  The proof of this theorem also allows us to construct examples of
model categories which are not cofibrantly generated: see
Corollary~\ref{cor:non-cg}.  Moreover, we show that when the poset and weak
equivalences are especially nice the zigzag of equivalences can be taken to
consist only of the identity functor on the underlying category; see
Theorem~\ref{thm:modelsequiv}.  It would be highly desirable to be able to prove
that such zigzags exist in general, and the existence of notable special cases
(see for example \cite[Theorem 5.7]{dugger01} or \cite[Theorem 7.5]{bergner07})
shows that it ought to be possible.

This paper is organized as follows.  Section~\ref{sec:tech} contains technical
preliminaries on lifting systems, model structures, and the particular ways they
behave in posets.  Section~\ref{sec:centers}
introduces the notion of a center and explores the interactions of centers and
model structures.
Section~\ref{sec:construction} constructs a model structure given a choice of
centers and proves Theorem~\ref{thm:modelexistsall}.
Section~\ref{sec:countable} provides an alternate characterization of the
existence of model structures on countable posets and proves
Theorem~\ref{thm:second-order}.  Lastly, Section~\ref{sec:comp} compares
different model structures extending a given pair and shows that in many cases
all such model structures are equivalent.

\subsection*{Notation}

All categories are assumed to be skeletal, in the sense that if $A \rto B$ is an
isomorphism in $\C$ then $A = B$.  As equivalence of categories preserves model
structures and all categories are equivalent to a skeletal category, this does
not lose any generality for our results.  A \emph{poset} is a skeletal
category $\C$ such that for all objects $A$ and $B$, $\#\Hom_\C(A,B) \leq 1$.
When $\C$ is small then it uniquely defines a poset in the classical sense, with
underlying set $\ob\C$ and relation $A \geq B$ if $\#\Hom_\C(A,B) = 1$.
Conversely, given a classical poset $P$ we can define a category $\C$ with
$\ob\C = P$ and $\Hom_\C(A,B) = \{*\}$ if $A \geq B$ and $\emptyset$ otherwise.
Thus our notion of a poset corresponds exactly to the classical notion of a
poset except that we allow the class of objects to be a proper class, not simply
a set.\footnote{To be completely consistent we may want to use the word
  ``poclass'' instead of ``poset'' to emphasize this fact, but as ``poclass'' is
  a much more nonstandard term we avoid its usage.}

In a poset, for any diagram
\[B \longleftarrow A \longrightarrow C\]
the pushout $B\cup_A C$ is equal to $B\cup C$.  For concision we
write both of these as $B \cup C$. Dually, we write $Y\times Z$ for
$Y\times_X Z$.

A category $\C$ is \emph{finitely bicomplete} if it contains all finite limits
and colimits; it is \emph{bicomplete} if it contains all small limits and
colimits.

\subsection*{Acknowledgements}

The authors would like to thank Jonathan Campbell and Wesley Calvert for their
thoughts on the paper, as well as the anonymous referee whose comments on the
exposition (including the definitions of ``semi-(co)fibrant'' and $\W^\chi_f$)
greatly improved the paper.  Zakharevich was supported in part by NSF grant
DMS-1654522.

\section{Lifting systems, model structures, and posets} \label{sec:tech}

We begin by recalling the definition of maximal lifting system and weak
factorization system.  For more background on these, especially in relation to
model categories, see for example \cite[Chapter 14]{mayponto} or \cite[Section
11]{riehl14}.

\begin{definition}
  For any two morphisms $f:A \rto B$ and $g:X \rto Y$ in $\C$ we say that
  \emph{$f$ lifts on the left of $g$} or \emph{$g$ lifts on the right of
    $f$} if for all commutative squares
  \[\xymatrix{A \ar[r] \ar[d]_f & X \ar[d]^g \\ B \ar[r] & Y}\]
  there exists a morphism $h:B \rto X$ which makes the diagram commute.  If $f$
  lifts on the left of $g$ we write $f\llp g$.

  For any class $S$ of morphisms of $\C$, we write 
  \begin{align*}
    S^\llp &= \{g\in \C\,|\, f\llp g\hbox{ for all }f\in S\}, \hbox{ and} \\
    {}^\llp S &= \{f\in \C\,|\, f\llp g\hbox{ for all }g\in S\}. 
  \end{align*}
  Note that both $S^\llp$ and $^\llp S$ can be proper classes.
\end{definition}

\begin{definition} \label{def:mls}
  A \emph{maximal lifting system} (henceforth written MLS) in $\C$ is a pair
  of classes of morphisms $(\LL, \RR)$ satisfying the following three
  conditions:
  \begin{enumerate}
  \item $\LL \llp \RR$.
  \item ${}^\llp\RR \subseteq \LL$.
  \item $\LL^\llp \subseteq \RR$.
  \end{enumerate}
  A \emph{weak factorization system} (henceforth written WFS) is a MLS such that
  every morphism $f$ in $\C$ can be factored as $f_Rf_L$ with $f_R\in \RR$ and
  $f_L\in \LL$.
\end{definition}

\begin{lemma} \label{lem:liftclass} Let $J$ be any class of morphisms in $\C$.
  Then $J^\llp$ is closed under composition, pullbacks in $\C$ and arbitrary
  products.  Dually, $^\llp J$ is closed under composition, pushouts in $\C$ and
  arbitrary coproducts.
\end{lemma}

For a proof, see for example \cite[14.1.8]{mayponto}.

We now recall the definition of a model category, using the WFS definition (as
presented in, for example, \cite{mayponto} and \cite{riehl14}).  

\begin{definition}
  \label{shorterAxioms}
  A \emph{model structure} $\CC$ on a finitely bicomplete category $\C$ is the
  specification of three subcategories of $\C$ called the \emph{weak
    equivalences} ($\CC_{we}$), the \emph{cofibrations} ($\CC_{cof}$) and the
  \emph{fibrations} ($\CC_{fib}$). Those three subcategories should respect the
  following axioms.
  \begin{description}
  \item[WFS]  The pairs 
    \[(\CC_{cof},\CC_{fib}\cap \CC_{we}) \qquad (\CC_{cof}\cap
    \CC_{we},\CC_{fib})\] are WFSs.
  \item[2OF3] For morphisms $f$ and $g$, if two of the
    morphisms $f$, $g$ and $gf$ are weak equivalences, then so is the third.
  \end{description}
  We call a morphism which is both a cofibration (resp. fibration) and a weak
  equivalence an \emph{acyclic cofibration} (resp. \emph{acyclic fibration}).
  An object $A$ such that the morphism $\initial \rto A$ is a cofibration
  (resp. fibration) is called \emph{cofibrant} (resp. \emph{fibrant}.  An object
  which is both cofibrant and fibrant is called \emph{bifibrant}. We call any
  connected component of $\CC_{we}$ a \emph{weak equivalence class}.
\end{definition}

\begin{remark}
  The definition above is an equivalent restatement of Quillen's original
  definition of a closed model category.  In more modern treatments it is
  customary to assume that $\C$ is bicomplete, not finitely bicomplete, as the
  construction of factorizations generally requires small limits and colimits,
  not just finite ones.  In Section~\ref{sec:comp} we will need this assumption
  to compare model structures.  However, for the main theorem in this paper this
  assumption is counterproductive, since the existence of small limits and
  colimits is not a first-order assumption.  However, the existence of
  \emph{finite} limits and colimits is, since it only requires the existence of
  an initial object, a terminal object, binary (co)products and (co)equalizers.
\end{remark}

From this point onwards, $\C$ is a finitely bicomplete poset.  We begin with a
lemma which is used repeatedly to prove lifting properties.

\begin{lemma} \label{lem:see-lift} Let $J$ be a class of morphisms in $\C$,
  closed under pushouts along morphisms in $\C$.  Then $J \llp f$ if and only if
  for all factorizations of $f:A \rto B$ as $A \stackrel{f'}\rto C \rto B$, if
  $f'\in J$ then $f' = 1_A$.
\end{lemma}

\begin{proof}
  First, suppose that $J\llp f$ and consider any factorization of $f$ as
  $A \stackrel{f'}\rto C \rto B$ where $f'\in J$.  We then have a diagram
  \[\xymatrix{A \ar[r]^= \ar[d]_{f'} & A \ar[d] \\ C \ar[r] & B}\]
  which must have a lift; thus $A = C$.

  Conversely, suppose that the condition in the lemma holds, and consider any
  diagram
  \[\xymatrix{X \ar[r] \ar[d]_g & A \ar[d]^f \\ Y \ar[r] & B}\]
  with $g\in J$.  As $J$ is closed under pushouts, the morphism
  $g':A \rto A\cup Y$ is also in $J$.  Since $f$ factors through $g'$ and
  $g'\in J$ we must have $A\cup Y = A$.  Thus the morphism $Y \rto A\cup Y = A$
  is a lift in the diagram, and $J\llp f$.
\end{proof}

We now turn to a uniqueness lemma.

\begin{lemma} \label{uniqueness} In $\CC$, factorizations into an acyclic
  cofibration and a fibration or a cofibration and an acyclic fibration are
  unique.
  
  Each weak equivalence class has a unique fibrant and cofibrant object.  In
  addition, in each weak equivalence class all elements in the class are at zigzag
  distance at most two from this object.  The zigzags can be chosen to consist
  of an inverse acyclic fibration and an acyclic cofibration.
\end{lemma}
\begin{proof}
  Let $f:A \rto B$ be any morphism, and consider two factorization of $f$:
  \[\xymatrix{ A \ar@{^{ (}->}[r] \ar@{^{ (}->}[d] & B' \ar@{->>}[d]^\sim \\ B''
      \ar@{->>}[r]^\sim & B}.\]
  Since cofibrations lift against acyclic fibrations, there exist morphisms $B'
  \rto B$ and $B \rto B'$.  Since $\C$ is a poset these must both be identities,
  and the factorization is unique.  The statement for factorizations into an
  acyclic cofibration followed by a fibration follows analogously.

  Suppose that $X$ and $Y$ are two bifibrant objects which are in the same weak
  equivalence class.  Since they are isomorphic in $\Ho \CC$, there exist
  morphisms $X \rto Y$ and $Y \rto X$ in $\C$; since $\C$ is a poset these must
  be identities, and $X = Y$.  Thus each weak equivalence class contains a
  unique bifibrant object.

  Now suppose that $A$ is any object.  Then there is a diagram
  \[\xymatrix{A \ar@{^{ (}->}[d]_\sim& A^c \ar@{->>}[l]_\sim \ar@{^{
          (}->}[d]_\sim \\ A^f
       & A^{cf} \ar@{->>}[l]_\sim}\]
  where $A^c$ is a cofibrant replacement for $A$, $A^f$ is a fibrant replacement
  for $A$, and $A^{cf}$ is both a cofibrant replacement for $A^f$ and a fibrant
  replacement for $A^c$ (which will end up being equal because there is a unique
  bifibrant object in the weak equivalence class of $A$).  This constructs the
  length-two zigzags.  
\end{proof}

The following condition on a subclass of morphisms is a strengthening of the
usual 2-of-3 property for weak equivalences.

\begin{definition}
  We say that a class $\mathcal{E}$ of morphisms in $\C$ is \emph{decomposable}
  if for any morphism $f\in \mathcal{E}$, if $f = gh$ for some morphisms $g$ and
  $h$, then both $g$ and $h$ are in $\mathcal{E}$.
\end{definition}

\begin{proposition} \label{s2of3} $\CC_{we}$ is decomposable.
\end{proposition}

\begin{proof}
  Fix $f:A \stackrel{\sim}{\rto} B$ in $\CC_{we}$.  Write $f = hg$.  Factor $g$
  as a cofibration followed by an acyclic fibration, and factor $f$ as an
  acyclic cofibration followed by an acyclic fibration, as illustrated in the
  following diagram:
  \[\xymatrix{A \cofib[rr]^{f_{ac}}_\sim \cofib[d]_{g_c} &&  A' \fib[d]^{f_{af}}_\sim \\
    C' \fib[r]^{g_{af}}_\sim & C \ar[r]^h & B}\]
  Then this diagram has a lift $\alpha: C' \rto A'$.  As $\C$ is a poset,
  $\alpha$ is the pushout of $f_{ac}$ along $g_c$, so it must also be an acyclic
  cofibration.  By (2OF3) $g_c$ is also a weak equivalence.  Thus
  $g$ is also a weak equivalence, and by (2OF3) $h$ is as well.
\end{proof}

We mention an important example of a particular type of weak equivalence class.

\begin{ex} \label{ex:butterfly}
  Suppose that $\CC$ contains a seven-object weak equivalence class with the
  following diagram (and no other morphisms between these seven objects):
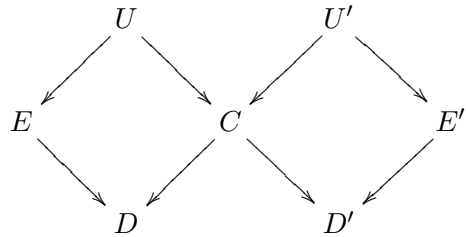
  \[\xymatrix{& U \ar[ld] \ar[rd] & & U' \ar[ld] \ar[rd] \\ 
    E \ar[rd]  & & C \ar[ld] \ar[rd]  & & E' \ar[ld] \\ 
    & D & & D'
  }\]
  Then the model structure must assign the morphisms as follows:
  
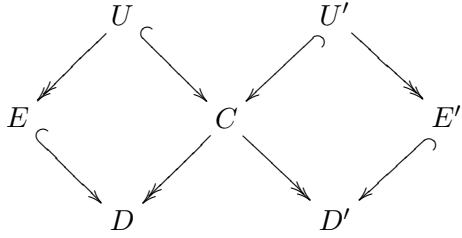
\[\xymatrix{& U \fib[ld] \cofib[rd] & & U' \cofib[ld] \fib[rd] \\ 
    E \cofib[rd]  & & C \fib[ld] \fib[rd]  & & E' \cofib[ld] \\ 
    & D & & D'
  }\]
  $C$ must be the cofibrant fibrant object, as it is the only object with zigzag
  distance $2$ from all other objects in the weak equivalence class.  $U$ and
  $U'$ must be cofibrant, as they receive no weak equivalences; dually, $D$ and
  $D'$ must be fibrant.  The morphisms $U \rto C$ and $C \rto D$ are
  cofibrations and fibrations, respectively, as $U,U'$ cannot be fibrant and
  $D,D'$ cannot be cofibrant. By Proposition \ref{s2of3}, the morphism $U
  \rto E$ is the pullback of the morphism $C \rto D$ along $E \rto D$, so it is
  also a fibration; dually, the morphism $E \rto D$ is the pushout of $U \rto C$
  and must be a cofibration.
\end{ex}

Bifibrant objects are vitally important to model structures, as they are ``good
choices'' for both mapping into and mapping out of.  In a poset the choice of
bifibrant objects is uniquely functorial, and thus these give a ``good'' retract
of the category.

\begin{proposition} \label{prop:bifibfunc}
  The map $C \mapsto C^{cf}$ extends to a functor $\C \rto \C$.
\end{proposition}

This observation is the key to the definition of a \textsl{center}, given in the
following section.  We finish up this section with two technical observations which
motivate the definition of the model structure in Section~\ref{sec:construction}.

\begin{proposition} \label{prop:inCofib} If $B$ is any cofibrant object in $\CC$
  and $f:A \rto B$ is any morphism in $\C$, then $f$ is a cofibration in $\CC$.
  Dually, if $A$ is fibrant then $f$ is a fibration.
\end{proposition}
\begin{proof}
  Factor $f$ into a cofibration followed by an acyclic fibration and consider
  the following diagram:
  \[\xymatrix{\initial \ar[r] \cofib[d] & A \cofib[r] \ar[rd]_f & A' \fib[d]^\sim \\
    B \ar[rr]^= & & B}\]
  By (WFS) this has a lift $B \rto A'$.  As $\C$ is a poset we conclude that $B
  = A'$, so $f$ is equal to the cofibration $A \hookrightarrow A'$.  The second
  part follows by duality.
\end{proof}
\begin{corollary}
  \label{comparisonToCenter}
  If $C = C^{cf}$ and $f:U \rto C$ is in $\CC_{we}$ then $f$ is an acyclic
  cofibration. Dually, if $g:C \rto D$ is in $\CC_{we}$ then $g$ is an acyclic
  fibration.
\end{corollary}

\section{Centers} \label{sec:centers}

We now turn to encoding properties of bifibrant objects in a more direct manner.
Inspired by Proposition~\ref{prop:bifibfunc} we define a ``center'' of a weak
equivalence class to be given by a choice of retraction which is compatible with
weak equivalences.  Such a retraction will encode all of the relevant properties
of bifibrant objects and will allow us to construct a model structure.  For the
rest of this section, fix a finitely bicomplete poset $\C$ and a subcategory
$\W$ that is decomposable.  We denote morphisms in $\W$ by
$\stackrel{\sim}{\rto}$.

\begin{definition}
  A \emph{choice of centers} is a functor
  \[\chi: \C \rto \C\]
  such that the following properties hold:
  \begin{description}
  \item[C1] The image of $\chi|_\W$ only contains identity morphisms.
  \item[C2] For all $A\in \C$ the diagram
    \[\xymatrix{A\times \chi(A) \ar[d] \ar[r] & \chi(A) \ar[d] \\ A \ar[r] &
      A\cup \chi(A)}\]
    lies in $\W$.
  \end{description}
\end{definition}

Condition (C1) implies that if $f:A \rto B$ is in $\W$ then
$\chi(f) = 1_{\chi(A)}$.  In particular, if there exists a zigzag of morphisms
in $\W$ connecting $A$ and $B$ then $\chi(A) = \chi(B)$.  In particular, $\chi$
must be idempotent: $\chi(\chi(A)) = \chi(A)$.

We can now make our claim that centers are akin to bifibrant objects precise by
showing that any model structure produces a choice of centers by taking any
objects to its bifibrant approximation.

\begin{lemma} \label{lem:model->centers}
  Every model structure on $\C$ gives a choice of centers.
\end{lemma}

\begin{proof}
  Let $\CC$ be any model structure on $\C$.  We define $\chi(A) = A^{cf}$, the
  bifibrant object in the same weak equivalence class as $A$; this is unique by
  Lemma~\ref{uniqueness} so $\chi$ is well-defined and satisfies the first
  condition for a choice of centers.  To check the second one, let $A^c$ be a
  cofibrant replacement of $A$ and $A^f$ be a fibrant replacement of $A$; then
  we have a diagram
  \[\xymatrix{A^c \cofib[r]^-\sim \fib[d]_\sim & \chi(A) \fib[d]^\sim \\ A
    \cofib[r]^-\sim & A^f}\]
  in $\CC_{we}$.  By Proposition~\ref{s2of3}, $\CC_{we}$ is decomposable, so the
  square 
  \[\xymatrix{A\times \chi(A) \cofib[r]^-\sim \fib[d]_\sim & \chi(A) \fib[d]^\sim \\ A
    \cofib[r]^-\sim & A\cup \chi(A)}\]
  must also be in $\CC_{we}$.
\end{proof}

Even though $\chi$ is uniquely determined by $\CC$, the model structure $\CC$ is
not uniquely determined by $\chi$.

\begin{ex}
  The following two model structures have the same choice of centers.  All
  cofibrant objects (other than $\emptyset$) are marked with $\cdot^c$ and all
  fibrant objects (other than $*$) are marked with $\cdot^f$.
  \[\xymatrix@C=1.7em@R=1.7em{& & B^c \cofib[rd]^\sim \\
    \emptyset \cofib[r] & A^c \cofib[ru]^\sim \cofib[rd]_\sim \cofib[rr]^\sim& & C^{cf} \fib[r] & {*} \\
    & & {B'}^c \cofib[ru]_\sim } 
  \qquad
  \xymatrix@C=1.7em@R=1.7em{& & B \cofib[rd]^\sim \\
    \emptyset \cofib[r] & A^c \fib[ru]^\sim \fib[rd]_\sim \cofib[rr]^\sim & & C^{cf} \fib[r] & {*} \\
    & & B' \cofib[ru]_\sim }\]
\end{ex}

Just as bifibrant objects record the homotopical information in a model
structure, the choice of centers records homotopical information in a poset.  In
particular, choices of centers identify the weak equivalences.

\begin{lemma} \label{lem:all-ob} Any morphism $f:A \rightarrow B$ in $\C$ such
  that $\chi(A) = \chi(B)$ is in $\W$.
\end{lemma}

In particular, this implies that in a model structure, any morphism between two
objects in the same weak equivalence class is itself a weak equivalence.

\begin{proof}
  Let $C = \chi(A) = \chi(B)$.  By (C2) the morphisms $A \times C \rto C$ and
  $C \rto B\cup C$ are both in $\W$.  Thus $A \times C \rto C \rto B\cup C$ is
  in $\W$.  But we can also factor this morphism as
  \[A \times C \longrightarrow A \stackrel{f}{\longrightarrow} B \longrightarrow B \cup C,\]
  so, since $\W$ is decomposable, $f$ is a weak equivalence.
\end{proof}

It is also the case that choices of centers are all closely related.

\begin{lemma} \label{lem:center-product}
  If $\chi_1$ and $\chi_2$ are choices of centers then $\chi_1\times
  \chi_2$ is also a choice of centers.  Dually, $\chi_1\cup \chi_2$ is also a
  choice of centers.
\end{lemma}

\begin{proof}
  We prove the first part; the second follows by duality.

  Since $\C$ is closed under products, $\chi_1\times \chi_2$ is clearly a
  well-defined functor $\C \rto \C$. We just need to check the other
  conditions.
  
  \noindent
  (C1) We need to show that $\chi_1\times\chi_2 |_\W$ hits only identity
  morphisms.  If $A \stackrel{\sim}{\rto} B$ then $\chi_1(A) = \chi_1(B)$ and
  $\chi_2(A) = \chi_2(B)$, so $\chi_1\times \chi_2(A) = \chi_1(B) \times
  \chi_2(B)$, as desired.

  \noindent
  (C2) We write $C_i = \chi_i(A)$ for $i=1,2$ in the interests of space.  We
  know that there exists a diagram
  \[\xymatrix@R=1em{
    & C_1\times A \ar[ld] \ar[rd] & & C_2\times A \ar[ld] \ar[rd]\\
    C_1 \ar[rd] && A \ar[ld] \ar[rd] && C_2 \ar[ld]\\
    & C_1 \cup A & & C_2\cup A }\] in $\W$; thus $C_1$ and
  $C_2$ are connected by a zigzag of morphisms in $\W$, and in
  particular we know that $\chi_1(C_2) = C_1$.  Thus we also have
  a diagram
  \[\xymatrix{C_1 \times C_2 \ar[r] \ar[d] & C_2 \ar[d] \\ C_1 \ar[r] &
    C_1 \cup C_2}\]
  in $\W$.  We want to show that the diagram
  \[\xymatrix{C_1\times C_2 \times A \ar[r] \ar[d] & C_1\times C_2 \ar[d] \\ 
    A \ar[r] & (C_1\times C_2) \cup A  }\]
  is is $\W$.  Note that $\chi_2(C_1\times A) = C_2$, so the
  morphism $(C_1 \times A)\times C_2 \rto C_1\times A$ is in
  $\W$.  Thus we have the following diagram,
  \[\xymatrix@C=2em{& (C_1\times C_2)\times A \ar[ld]_\sim  \ar[d]
    \ar[r] & C_1\times
    C_2 \ar[r]^-\sim \ar[d] & C_1 \ar[d]^\sim \\
    A\times C_1 \ar[r]^-\sim & A \ar[r] \ar@/_1.2pc/[rr]_\sim &
    (C_1 \times C_2)\cup A \ar[r] & C_1
    \cup A}
  \]
  where the morphisms that we know are in $\W$ are marked with $\sim$.  The
  fact that the middle square is in $\W$ follows because $\W$ is decomposable.
\end{proof}

To finish the discussion of centers we prove a technical lemma which will help
in the future for constructing WFS.  Classicaly, factorizations are constructed
using a small object argument in some fashion.  In our case we do not do this,
as we want to choose ``bifibrant generators'' rather than cofibrant generators.
It turns out that when we are working with a poset, rather than a more
complicated category, this is fairly straightforward.  To assist with clarity,
we introduce an extra definition.

\begin{definition}
  An object $A$ is defined to be \textsl{semi-fibrant}
  (resp. \textsl{semi-cofibrant}) if there exists a morphism $\chi(A) \rto A$
  (resp. $A \rto \chi(A)$).
\end{definition}

Directly from the definition it follows that any object of the form $\chi(A)
\times A$ (resp. $\chi(A) \cup A$) is semi-cofibrant (resp. semi-fibrant).  In
particular, $\chi(A)$ is both semi-fibrant and semi-cofibrant.

\begin{lemma} \label{lem:fact->wfs} Let $\chi$ be a choice of centers for $(\C,\W)$.
  Suppose that $(\LL, \RR)$ is a pair of classes of morphisms such that
  \begin{enumerate}
  \item \label{f->w1} Both $\LL$ and $\RR$ are closed under composition and $\LL \llp \RR$,
  \item \label{f->w2} $\LL$ is closed under pushouts along morphisms in $\C$ and
    $\RR$ is closed under pullbacks along morphisms in $\C$,
  \item \label{f->w4} All morphisms with semi-fibrant domain are in $\RR$ or all
    morphisms with semi-cofibrant codomain are in $\LL$, and
  \item \label{f->w5} All morphisms in $\W$ factor as a morphism in $\LL$ followed by a
    morphism in $\RR$.
  \end{enumerate}
  Then $(\LL, \RR)$ is a WFS.
\end{lemma}

\begin{proof}
  We prove this assuming that the first part of condition (\ref{f->w4}) holds.
  Since the other conditions are self-dual, the proof for the other part follows
  by duality.
  
  As $\LL \llp \RR$, if all morphisms in $\C$ factor as a morphism in $\LL$
  followed by a morphism in $\RR$ then by \cite[14.1.13]{mayponto} $(\LL, \RR)$
  is a WFS.  Consider any morphism $f:A \rto B$ in $\C$.  We can factor $f$ as
  \[A \stackrel{f'}{\longrightarrow} (A\cup \chi(A))\times B
  \stackrel{f''}{\longrightarrow} B;\] we claim that $f'$ is in $\W$ and $f''$
  is in $\RR$.  Then using condition (\ref{f->w5}) on $f'$ we can write $f' =
  f'_Rf'_L$ and the desired factorization is then
  \[f = \underbrace{f''f'_R}_{\in \RR} \underbrace{f'_L}_{\in \LL}.\]
  The morphism $A \rto A\cup \chi(A)$---which is in $\W$---factors as $A \rto
  (A\cup \chi(A))\times B \rto A\cup \chi(A)$, so since $\W$ is decomposable $f'$ is in $\W$.
  It remains only to check that $(A\cup \chi(A))\times B \rto B$ is in $\RR$.
  
  Because $\chi$ is a functor, there is a morphism $A \cup \chi(A) \rto B\cup\chi(B)$.
  By hypothesis (\ref{f->w4}), this morphism is in $\RR$;
  thus by hypothesis (\ref{f->w2}) its pullback along the morphism $B \rto
  B\cup\chi(B)$ must also be in $\RR$.  Thus $(A\cup\chi(A))\times B \rto B$ is
  in $\RR$.
\end{proof}

\section{Construction of model structures} \label{sec:construction} The goal of
this section is to prove Theorem~\ref{thm:modelexistsall}.  We therefore fix a
relative category $(\C,\W)$ and a choice of centers $\chi$ and use these to
construct a model structure.  As before, we assume that $\C$ is finitely
bicomplete and $\W$ is decomposable.

\begin{lemma} \label{lem:Jchigood}
  Let $\{A_i\}_{i\in I}$ be a family of semi-cofibrant objects such that
  $\coprod_{i\in I} A_i$ and $\coprod_{i\in I} \chi(A_i)$ exist.  Then
  $\coprod_{i\in I} A_i$ is also semi-cofibrant.  Dually, if $\{A_i\}_{i\in I}$
  is a family of semi-fibrant objects such that $\prod_{i\in I} A_i$ and
  $\prod_{i\in I} \chi(A_i)$ exist, then $\prod_{i\in I}A_i$ is semi-fibrant.
\end{lemma}

\begin{proof}
  We prove the first part of the lemma; the second follows by duality.  For all
  $i\in I$ there is a morphism $A_i \rto \coprod_i A_i$, and thus a morphism
  $A_i \rto \chi(A_i) \rto \chi(\coprod_iA_i)$.  Thus there exists a morphism
  $\coprod_i A_i \rto \chi(\coprod_i A_i)$, as desired.
\end{proof}

Recall that, in a poset, in any composition $A \stackrel{f}\rto B \stackrel{g}\rto
C$, $g$ is a pushout of $gf$ and $f$ is a pullback of $gf$.  Thus the
semi-fibrant and semi-cofibrant objects contain a lot of information about which
morphisms ``ought'' to be acyclic cofibrations/fibrations.

\begin{definition}
  Write $Q_\chi$ for the full subcategory of $\W$ with semi-fibrant domain and
  codomain, and $J_\chi$ for the full subcategory of $\W$ with semi-cofibrant
  domain and codomain.  
\end{definition}

Note that if $\W$ is decomposable then the class $Q_\chi$ is decomposable and
the class $J_\chi$ is decomposable.  In addition, a morphism in $\W$ with
semi-fibrant domain (resp. semi-cofibrant codomain) automatically has
semi-fibrant codomain (resp. semi-cofibrant domain).  

\begin{lemma} \label{lem:JchiQchi} Suppose $f:A \rto B$ is a morphism with $B$
  semi-cofibrant.  Then $f \llp Q_\chi$.  Dually, if $A$ is semi-fibrant then
  $J_\chi \llp f$.  In particular, $J_\chi\llp Q_\chi$.
\end{lemma}

In particular,  for any object $A$ in $\C$,
\[(\initial \rto \chi(A)) \llp Q_\chi \qquad\hbox{and}\qquad J_\chi \llp
(\chi(A) \rto *).\]

\begin{proof}
  We prove the first statement; the second follows by duality.  Let $p:X \rto Y \in Q_\chi$, and consider a diagram
  \[\xymatrix{A \ar[r] \ar[d]_f & X \ar[d]^p \\ B \ar[r] & Y}\]
  Applying $\chi$ to the square takes $p$ to the identity morphism on $\chi(X)$,
  and by the defining properties of $Q_\chi$ and $f$ we get a diagram
  \[\xymatrix{A \ar@/^3ex/[rrr] \ar[d]_f & & \chi(X) \ar[r]
    \ar[d]^= & X \ar[d]^p \\ 
    B \ar[r] & \chi(B) \ar[r] \ar@/_3ex/[ru] & \chi(Y) \ar[r] & Y }\]
  This gives the desired lift.
\end{proof}

We would like to identify those morphisms which ``behave like'' acyclic
cofibrations.  Acyclic cofibrations lift on the left of all fibrations;
Proposition~\ref{prop:inCofib} shows that, in a model structure on a poset, all
morphisms with fibrant domain are fibrations.  We thus take our definition of
acyclic cofibrations to be exactly those that lift on the left of the morphisms
with semi-fibrant domain.\footnote{We would like to extend our sincerest thanks to
  the anonymous referee, who pointed out this characterization and greatly
  simplified this portion of the exposition.}

\begin{definition}
  Let $\chi$ be a choice of centers.  We define
  \[\W^\chi_c = {}^\llp \{f:A \to B\in \C\,|\, A \hbox{ semi-fibrant}\}\]
  and
  \[\W^\chi_f = \{f:A \to B\in\C\,|\, B \hbox{ semi-cofibrant}\}^\llp.\]
  In particular $\W^\chi_c$ is closed under pushouts and $\W^\chi_f$ is closed
  under pullbacks.
\end{definition}

By Lemma~\ref{lem:JchiQchi} $J_\chi \subseteq \W_c^\chi$ and $Q_\chi \subseteq
\W_f^\chi$.  As implied by the notation, all morphisms in $\W_c^\chi$ and
$\W_f^\chi$ are weak equivalences:
\begin{lemma}
  \[\W_c^\chi \cup \W_f^\chi \subseteq \W.\]
\end{lemma}
\begin{proof}
  We prove that $\W_c^\chi \subseteq \W$; the result for $\W_f^\chi$ follows
  analogously.  Let $f: X \to Y$ be in $\W_c^\chi$.  Then it must lift on the
  left of $X\cup \chi(X) \to Y\cup \chi(Y)$.   In particular, $X \to X\cup
  \chi(X)$ factors through $f$; thus by decomposition $f$ is a weak
  equivalence, as desired. 
\end{proof}

We now have the following factorization result:

\begin{lemma} \label{lem:all-factor} Every morphism in $\W$ factors as a
  morphism in $\W^\chi_c$ followed by a morphism which is a pullback of a
  morphism in $Q_\chi$.
\end{lemma}

\begin{proof}
  Suppose that $X \rto Y$ is in $\W$ and let $C = \chi(X) = \chi(Y)$.  We claim that 
  \[X \to Y\times(X\cup C) \to Y\] is the desired factorization.  The morphism
  $Y\times (X \cup C) \to Y$ is a pullback of $X\cup C \to Y\cup C$, which is in
  $Q_\chi$.  It remains to show that $X \to Y\times(X\cup C)$ is in $\W^\chi_c$.
  Let $A \rto B$ be a morphism with $A$ semi-fibrant, and consider a square
  \[\xymatrix{ X \ar[r] \ar[d] & A \ar[d] \\ Y\times (X\cup C) \ar[r] & B }.\]
  To check that a lift exists it suffices to check that a morphism $Y\times (X
  \cup C) \rto A$ exists.  This is given by the composition
  $Y \times (X \cup C) \rto X\cup C \rto A$, where the morphism
  $X \cup C \rto A$ exists because the square gives a morphism $X \rto A$, and
  the fact that $C = \chi(X)$ gives a morphism $C = \chi(X) \rto \chi(A) \rto A$
  (since $A$ is semi-fibrant).
\end{proof}

We are now ready to construct a model structure that depends only on a choice of
centers.  The fibrations in this model structure are defined to be the na\"ive
set of fibrations making the semi-fibrant objects fibrant.

\begin{definition} \label{def:Cchi}
  Given a choice of centers $\chi$, the model structure $\CC^\chi$ is defined by 
  \[\CC^\chi_{we} = \W \qquad \CC^\chi_{fib} = (\W^\chi_c)^\llp \qquad \CC^\chi_{cof} =
  {}^\llp(\CC^\chi_{fib}\cap \CC^\chi_{we}).\]
\end{definition}

\begin{proposition} \label{prop:Cchi}
  $\CC^\chi$ is a model structure.
\end{proposition}

\begin{proof}
  We need to show that $(\CC^\chi_{cof} \cap \CC^\chi_{we}, \CC^\chi_{fib})$ and
  $(\CC^\chi_{cof}, \CC^\chi_{fib}\cap \CC^\chi_{we})$ are WFSs.  We will use
  Lemma~\ref{lem:fact->wfs} for both, and check the conditions simultaneously.  

  \noindent (\ref{f->w1}) $\CC_{we}^\chi$ is closed under composition by
  definition; $\CC_{cof}^\chi$ and $\CC_{fib}^\chi$ are defined by lifting
  properties, and thus are closed under composition by
  Lemma~\ref{lem:liftclass}.  The lifting condition holds by definition for
  $(\CC^\chi_{cof}, \CC^\chi_{fib}\cap \CC^\chi_{we})$.  To prove the lifting
  condition for $(\CC^\chi_{cof}\cap \CC^\chi_{we}, \CC^\chi_{fib})$ it suffices
  to show that $\CC^\chi_{cof} \cap \CC^\chi_{we} \subseteq \W^\chi_c$.  Let $f$
  be in $\CC^\chi_{cof}\cap \CC^\chi_{we}$.  By Lemma~\ref{lem:all-factor} we
  can write $f = f_rf_c$ with $f_c \in \W^\chi_c$ and $f_r$ a pullback of a
  morphism in $Q_\chi$.  Every morphism in $Q_\chi$ is in
  $\W^\chi_f\cap (\W^\chi_c)^\llp \subseteq \CC^\chi_{we}\cap \CC^\chi_{fib}$.
  Thus we have a diagram
  \[\xymatrix{\bullet \cofib[r]^{f_c} \cofib[d]_f & \bullet \fib[d]^{f_r}_\sim \\ \bullet
    \ar[r]^= & \bullet}\]
  which has a lift because $f$ is in $\CC^\chi_{cof}$.  Thus $f = f_c \in
  \W^\chi_c$, as desired.

  \noindent(\ref{f->w2}) First consider
  $(\CC^\chi_{cof}, \CC^\chi_{fib} \cap \CC^\chi_{we})$.  We have
  $\CC^\chi_{cof} = {}^\llp(\CC^\chi_{fib}\cap \CC^\chi_{we})$, so it is
  automatically closed under pushouts.  Now let $f$ be in
  $\CC^\chi_{fib} \cap \CC^\chi_{we}$.  Since by definition $\CC^\chi_{fib}$ is
  closed under pullbacks, it suffices to show that $f\in \W^\chi_f$, which is
  closed under pullbacks by definition.  By Lemma~\ref{lem:all-factor} we can
  factor $f$ as $f_2f_1$, with $f_2\in \W^\chi_f$ and $f_1\in \W^\chi_c$.  Then
  we have the following diagram:
  \[\xymatrix{\bullet \ar[r]^= \ar[d]_{f_1} & \bullet  \ar[d]^f \\ \bullet
    \ar[r]^{f_2} & \bullet}\]
  Since $f$ is in $\CC^\chi_{fib} = (\W^\chi_c)^\llp$, a lift exists in this
  diagram, and we see that $f = f_2$ which is in $\W^\chi_f$, as desired.

  Second consider $(\CC^\chi_{cof}\cap \CC^\chi_{we}, \CC^\chi_{fib})$.  By
  definition we know that $\CC^\chi_{cof}$ is closed under pushouts.  Since
  $\CC^\chi_{cof} \cap \CC^\chi_{we} \subseteq \W^\chi_c$ we know that the pushout of
  any morphism in $\CC^\chi_{cof}\cap \CC^\chi_{we}$ is a weak equivalence, and
  thus $\CC^\chi_{cof}\cap \CC^\chi_{we}$ is closed under pushouts.
  $\CC^\chi_{fib}$ is closed under pullbacks by construction.

  \noindent
  (\ref{f->w4}) The condition is satisfied for
  $(\CC^\chi_{cof} \cap \CC^\chi_{we}, \CC^\chi_{fib})$ by the definition of
  $\CC^\chi_{fib}$. Now consider
  $(\CC^\chi_{cof}, \CC^\chi_{fib}\cap \CC^\chi_{we})$.  We show that all
  morphisms $f:A \rto B$ with $B$ semi-cofibrant lift on the left of
  $\CC^\chi_{fib} \cap \CC^\chi_{we}$, and thus are in $\CC^\chi_{cof}$.
  Consider a factorization of $f$ as $A \rto C \rto B$ with $C \rto B$ in
  $\CC^\chi_{fib} \cap \CC^\chi_{we}$.  Since $C\rto B\in \CC^\chi_{we}$,
  $\chi(C) = \chi(B)$ and thus $C \rto B$ is in $J_\chi \subseteq \W_c^\chi$.
  On the other hand, $C \rto B \in \CC^\chi_{fib} = (\W_c^\chi)^\llp$, so
  $C = B$.  Thus by Lemma~\ref{lem:see-lift},
  $f\llp (\CC^\chi_{fib} \cap \CC^\chi_{we})$.

  \noindent (\ref{f->w5}) By Lemma~\ref{lem:all-factor} all morphisms in $\W$
  factor as a morphism in $\W^\chi_c$ followed by a morphism which is a pullback
  of a morphism in $Q_\chi$.  As
  $\W^\chi_c\subseteq \CC^\chi_{cof}\cap \CC^\chi_{we}$, which was shown in
  (\ref{f->w1}) of this proof, and pullbacks of morphisms in $Q_\chi$ are in
  $\CC^\chi_{we} \cap \CC^\chi_{fib}$ the condition is satisfied for both WFSs.
\end{proof}

By duality we have the following.
\begin{corollary} \label{cor:chiC} Suppose that $\C$ is a finitely bicomplete
  poset, $\W$ is decomposable, and $\chi$ is a choice of centers.  Then the
  structure $^\chi\CC$ defined by
  \[^\chi\CC_{we} = \W \quad {}^\chi \CC_{cof} = {}^\llp (\W^\chi_f) \qquad {}^\chi
  \CC_{fib} = (^\chi\CC_{cof} \cap {}^\chi\CC_{we})^\llp\]
  is a model structure on $\C$.
\end{corollary}

\begin{remark} 
  Before this section, all definitions and results that we have discussed have
  been self-dual. The model structures constructed in this section are not and
  this asymmetry is unavoidable. It arises even when both $\C$, $\W$ and the
  choice of centers are self-dual.

  The following preorder on
  $5$ objects with every morphism considered a weak equivalence provides an
  example.
  
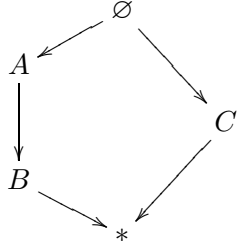
\[\xymatrix@R=.7em{
      & \initial \ar[rdd]\ar[ld] \\
      A \ar[dd]  \\
      && C\ar[ldd]  \\
      B \ar[rd]\\
      & \terminal \\}\] The object $C$ is chosen as center.  Then $\initial$ is
  the only semi-cofibrant object and $*$ is the only semi-fibrant object, and
  the morphism $A \rto B$ is in both $\W^\chi_c$ and $\W^\chi_f$.  In any model
  structure on the category, this morphism must be either an acyclic cofibration
  or an acyclic fibration---but not both!---breaking the symmetry.
\end{remark}

We are ready to prove Theorem~\ref{thm:modelexistsall}.

\begin{proof}[Proof of Theorem~\ref{thm:modelexistsall}]
  By Lemma~\ref{lem:model->centers}, any model structure gives a choice of
  centers. By Proposition~\ref{prop:Cchi} a
  choice of centers gives rise to at least one model structure. 
\end{proof}

\section{Model structures on countable posets} \label{sec:countable}

In this section we restrict our attention to pairs $(\C,\W)$ where $\C$ is
countable and $\W$ is decomposable, and show that in this case we can give a
first-order characterization of those pairs that extend to a model structure.

\begin{definition}
  let $W$ be a weak equivalence class in $\C$.  A \emph{proto-center} $P$ for
  $W$ is an object in $W$ such that for all $X\in W$, $X\times P \rto X$ and
  $X \rto X \cup P$ are weak equivalences.  For an object $A$, a
  \emph{proto-center for $A$} is a proto-center in the weak equivalence class
  of $A$.

  A proto-center $P$ is \emph{locally compatible} if 
  \begin{enumerate}
  \item for any morphism $A' \rto A$ such that $A$ is in the same weak
    equivalence class as $P$ there exists a morphism $P' \rto P$ where $P'$ is
    a proto-center in the weak equivalence class of $A'$, and
  \item for any morphism $A \rto A'$ such that $A$ is in the same weak
    equivalence class as $P$ there exists a morphism $P \rto P'$ where $P'$ is
    a proto-center in the weak equivalence class of $A'$.
  \end{enumerate}
\end{definition}

The following lemma shows that proto-centers locally behave the way choices of
centers do: the product of two proto-centers is a proto-center and so is the
coproduct.  (For comparison, see Lemma~\ref{lem:center-product}.)

\begin{lemma} \label{lem:prodproto}
  The set of proto-centers of a weak equivalence class is closed under binary
  products and coproducts.
\end{lemma}

\begin{proof}
  We prove that the product of two proto-centers is a proto-center; the
  closure by coproduct follows by duality.

  Let $P_1$ and $P_2$ be two proto-centers in a weak equivalence class $W$, and
  consider $Q = P_1 \times P_2$.  We need to show that for any $X\in W$,
  $X\times Q \rto X$ and $X \rto X\cup Q$ are in $\W$.  The morphism $X \times Q
  \rto X$ factors as
  \[(X \times P_1) \times P_2 \rto X\times P_1 \rto X.\]
  The first of these is in $\W$ because $P_2$ is a proto-center, and the second
  is in $\W$ because $P_1$ is a proto-center.  
  
  Now consider $X \cup Q$.  The morphism $X \rto X\cup P_1$ is in $\W$, since
  $P_1$ is a proto-center; but this morphism factors as $X \rto X\cup Q \rto
  X\cup P_1$.  Since $\W$ is decomposable each of these must be a weak
  equivalence and we have $X \rto X \cup Q \in \W$.
\end{proof}

\begin{lemma} \label{lem:prodlocproto} Let $W$ and $W'$ be weak equivalence
  classes, and suppose that there exists $f:A \rto A'$ with $A\in W$ and
  $A'\in W'$.  For any two locally compatible proto-centers $Q\in W$ and
  $Q'\in W'$, $Q\times Q'$ is a locally compatible proto-center in $W$ and
  $Q\cup Q'$ is a locally compatible proto-center in $W'$.
\end{lemma}

\begin{proof}
  We begin by showing that $Q\times Q'$ and $Q \cup Q'$ are proto-centers in the
  appropriate weak equivalence classes.  We prove only the statement for
  $Q\times Q'$; the second statement follows by duality.
  
  First, consider the morphism $Q \times Q' \rto Q$; we wish to show that this
  is in $\W$.  Since $Q'$ is locally compatible there exists a proto-center
  $P\in W$ and a morphism $P \rto Q'$.  By Lemma~\ref{lem:prodproto} $P\times Q$
  is also a proto-center, and thus $P\times Q \rto Q$ is in $\W$. The morphism
  $P \times Q \rto Q$ factors as $P\times Q \rto Q\times Q' \rto Q$; thus since
  $\W$ is decomposable, $Q \times Q' \rto Q$ is in $\W$.

  We now check that $Q \times Q'$ is a proto-center.  Let $A \in W$, and
  consider $A \times Q \times Q' \rto A$.  We have a composition
  \[A \times Q \times P \rto A \times Q \times Q' \rto A,\] which is in $\W$
  because $Q \times P$ is a proto-center.  Since $\W$ is decomposable,
  $A \times Q \times Q' \rto A$ is in $\W$, as desired.  Now consider
  $A \rto A \cup (Q \times Q')$.  The morphism $A \rto A\cup Q$ (which is in
  $\W$) factors through $A \cup (Q \times Q')$; thus it is in $\W$,
  as desired.
 
  We now need to check local compatibility of $Q \times Q'$; the statement for
  $Q \cup Q'$ follows by duality.  Let $B'' \rto B$ be any morphism with
  $B\in W$; let $W''$ be the weak equivalence class of $B''$.  Since $Q$ is a
  locally compatible proto-center there exists a proto-center $P''$ in $W''$
  with a morphism $P'' \rto Q$.  Thus there exists a morphism
  $P'' \rto Q \cup Q'$.  Since $Q'$ is locally compatible there exists a
  proto-center $R''\in W''$ with a morphism $R'' \rto Q'$.  Then
  $P'' \times R''$ is a proto-center in $W''$; since there exist morphisms
  $P'' \rto Q$ and $R'' \rto Q'$ there exists a morphism
  $P''\times R'' \rto Q\times Q'$, as desired.

  Now suppose that $B \rto B''$ is any morphism with $B\in W$; let $W''$ be the
  weak equivalence class of $B''$.  Since $Q$ is a locally compatible
  proto-center there exists a proto-center $R''\in W''$ with a morphism
  $Q \rto R''$.  Then $Q\times Q' \rto Q \rto R''$ gives the desired morphism.
\end{proof}

The existence of locally compatible proto-centers implies that there is a
well-defined ordering on weak equivalence classes.

\begin{lemma} \label{lem:ordering} Suppose that $W$ and $W'$ are distinct weak
  equivalence classes containing locally compatible proto-centers $Q\in W$ and
  $Q'\in W'$.  If there exists a morphism $A \rto A'$ with $A\in W$ and
  $A'\in W'$ then there does not exist a morphism $B' \rto B$ with $B'\in W'$
  and $B\in W$.
\end{lemma}

\begin{proof}
  Suppose that both $A \rto A'$ and $B' \rto B$ exist.  Then by
  Lemma~\ref{lem:prodlocproto} applied to $A \rto A'$, $Q \times Q' \in W$. On
  the other hand, by Lemma~\ref{lem:prodlocproto} applied to $B' \rto B$,
  $Q\times Q'\in W'$.  Thus $W \cap W' \neq \emptyset$, a contradiction.  Thus
  both $A \rto A'$ and $B' \rto B$ cannot exist.
\end{proof}

The point of locally compatible proto-centers is that they can be used to
construct approximations to choices of centers.  

\begin{definition}
  A \emph{partial choice of centers} is a functor
  $\tilde \chi: \tilde \C \rto \C$ such that the following properties hold:
  \begin{description}
  \item[PC1] $\tilde \C$ is a full subcategory of $\C$, and if $A$ and $A'$ are
    in the same weak equivalence class and $A\in \tilde \C$ then $A'\in \tilde \C$.
  \item[PC2] The image of $\tilde \chi|_{\W\cap \tilde \C}$ only contains identity morphisms.
  \item[PC3] $\tilde \chi(A)$ is a proto-center for $A$ for all $A\in \tilde \C$.
  \end{description}
\end{definition}
In particular, a partial choice of centers with $\tilde \C = \C$ is a choice of centers.

When we are given a partial choice of centers and a locally compatible
proto-center we can use the proto-center to extend the partial choice of
centers.  We encode the conditions for doing so in the following lemma.

\begin{lemma} \label{lem:pcoc-extend}
  Let $\tilde \chi:\tilde\C \rto \C$ be a partial choice of centers and let $Q$
  be a locally compatible proto-center for a weak equivalence class
  $W\subseteq \W$ which is not in $\tilde \C$.  Suppose that the following two
  conditions hold:
  \begin{enumerate}
  \item For all $A\in \tilde \C$ and $A'\in W$, if there exists a morphism
    $A \rto A'$ then there exists a morphism $\tilde\chi(A) \rto Q$.
  \item for all $A\in \tilde \C$ and $A'\in W$, if there exists a morphism
    $A' \rto A$ then there exists a morphism $Q \rto \tilde\chi(A)$.
  \end{enumerate}
  Then the functor
  \[\tilde\chi'(A) =
  \begin{cases}
    \tilde\chi(A) & \hbox{if } A\in \tilde \C \\
    Q & \hbox{if } A \in W
  \end{cases}\]
  defined on the full subcategory of $\C$ generated by $\tilde\C$ and $W$ is a
  partial choice of centers.
\end{lemma}

\begin{proof}
  We first check that it is a functor.  We have defined it on objects.  To check
  that it is well-defined we must check that it takes morphisms to morphisms.
  For a morphism $A \rto A'$ in $\tilde \C$ it is well-defined because $\tilde
  \chi$ is well-defined.  Given any morphism $A \rto A'$ with $A\in \tilde \C$
  and $A'\in W$, $\tilde\chi'(A \rto A') = \tilde \chi(A) \rto Q$ exists by
  condition (1), thus $\tilde \chi'$ is well-defined on such morphisms.
  Analogously it is well-defined on morphisms $A' \rto A$ with $A \in \tilde \C$
  and $A'\in W$ by condition (2).  It is compatible with composition because all
  maps between posets which are well-defined on objects and morphisms are
  functors.  It satisfies the conditions to be a partial choice of centers by
  definition. 
\end{proof}

We now use the machinery we have built to construct a choice of centers out of
locally compatible proto-centers.

\begin{theorem}
  \label{countableCharacterizationTheorem}
  If each weak equivalence class of $\C$ has a locally compatible proto-center
  and there is only a countable number of weak equivalence classes then there
  exists a choice of centers.
\end{theorem}

\begin{proof}
  Let $\{W_i\}_{i=1}^\infty$ be an enumeration of the weak equivalence
  classes in $\C$; let $\C_n$ be the full subcategory of $\C$ containing
  $\bigcup_{i=1}^n W_i$.  In the interest of conciseness, we also define
  $\C_{n,m}$ for $m>n$ to be the full subcategory of $\C$ containing both $\C_n$
  and $W_m$.  

  We prove the following statement: for each $n\geq 0$ we can construct a pair
  \[\left(\tilde \chi_n\colon \C_n \rto \C, \{Q_m\}_{m=n+1}^\infty\right)\]
  where $\tilde \chi_n$ is a partial choice of centers and for each $m$, $Q_m$
  and $\tilde\chi_n$ satisfy the conditions of Lemma~\ref{lem:pcoc-extend}.  We
  construct these pairs in such a way so that for all $n' > n$,
  $\tilde\chi_{n'}(A) = \tilde\chi_n(A)$ for all $A\in \C_n$.  Using this
  sequence we then define a choice of centers $\chi:\C\rto \C$ by
  \[\chi(A) = \chi_n(A) \quad\hbox{if }A\in W_n.\]
  This will prove the theorem.

  For our base case $n=0$, we let $\tilde\chi_0:\emptyset\rto \C$ be the trivial
  map, and we let $\{Q_m\}_{m=1}^\infty$ be a choice of locally compatible
  proto-centers for each weak equivalence class.  These exist by assumption.
  
  Now consider a general $n$, and suppose that we are given
  $\tilde\chi_{n-1}:\C_{n-1} \rto \C$ and a sequence $\{Q_m\}_{m=n}^\infty$
  such that each $Q_m$ satisfies the conditions of Lemma~\ref{lem:pcoc-extend}.
  We let $\tilde\chi_n$ be the functor constructed in
  Lemma~\ref{lem:pcoc-extend} for $\tilde\chi_{n-1}$ and $Q_n$.  We then define
  the sequence $\{Q_m'\}_{m=n+1}^\infty$ by 
  \[Q_m' =
    \begin{cases}
      Q_m\times Q_{n} & \exists\ A_m \rto A_{n}\hbox{ with }A_m\in
      W_m\hbox{ and } A_{n}\in W_{n},\\
      Q_m \cup Q_{n} &   \exists\ A_{n} \rto A_{m}\hbox{ with }A_m\in
      W_m\hbox{ and } A_{n}\in W_{n}, \\
      Q_m & \hbox{otherwise}.
    \end{cases}\] These conditions are mutually exclusive by
  Lemma~\ref{lem:ordering}.  We need to check that this pair satisfies the
  conditions required by the inductive hypothesis.  In particular, all we need
  to check is that for all $m>n$, $\tilde\chi_n$ and $Q_m'$ satisfy the
  conditions of Lemma~\ref{lem:pcoc-extend}.

  $Q_m'$ is a locally compatible proto-center in $W_m$ by
  Lemma~\ref{lem:prodlocproto}.  Now suppose that $A\in \C_n$ and $A'$ is in
  $W_m$, and suppose that there exists a morphism $A \rto A'$.  If $A\in W_n$ we
  need to show that there exists a morphism $Q_n \rto Q_m'$; but by definition
  $Q_m' = Q_m \cup Q_n$, so this exists.  Now suppose that $A\in W_i$ for
  $i < n$.  By the inductive hypothesis there exists a morphism
  $\tilde\chi_n(A) \rto Q_m$.  When there does not exist a morphism
  $A_m \rto A_n$ (with $A_n\in W_n$ and $A_m\in W_m$) there exists a morphism
  $Q_m \rto Q_m'$, so there exists a morphism $\tilde\chi_n(A) \rto Q_m'$, as
  desired.  If such a morphism $A_m \rto A_n$ exists then
  $Q_m' = Q_n \times Q_m$, so it suffices to check that there exists a morphism
  $\tilde\chi_n(A) \rto Q_n$.  Since $Q_m$ is a locally compatible proto-center,
  there exists a proto-center $P_i\in W_i$ and a morphism $P_i \rto Q_m$.  There
  must also exist a proto-center $P_n\in W_n$ and a morphism $Q_m \rto P_n$.
  Thus there is a morphism $P_i \rto P_n$ which $\tilde\chi_n$ takes to
  $\tilde \chi_n(A) \rto Q_n$.  Thus condition (1) of
  Lemma~\ref{lem:pcoc-extend} holds.  Condition (2) holds by symmetry.
\end{proof}

Since the property of being a proto-center and the property of being a locally
compatible proto-center are first-order properties, we get the
following: 

\begin{corollary}
  The existence of a model structure extending $(\C,\W)$ when $\W$
  only has countably many weak equivalence classes is first order definable.
\end{corollary}

We are ready to tackle Theorem~\ref{thm:second-order}.  

\begin{proof}[Proof of Theorem~\ref{thm:second-order}]
  We construct a pair $(\P,\W)$ where $\P$ is an uncountable poset and $(\P,\W)$
  satisfies all of the conditions of
  Theorem~\ref{countableCharacterizationTheorem} other than the countability of
  $\P$, but which does not extend to a model structure.  By the downward
  L\"owenheim--Skolem theorem, the pair $(\P,\W)$ has an elementarily equivalent
  countable model $(\P',\W')$.  By
  Theorem~\ref{countableCharacterizationTheorem} there is a Quillen model
  structure extending $(\P',\W')$.  This gives two pairs with the same first
  order theory where only one extends to a Quillen model structure; the
  statement of the theorem follows.

  Let $\P$ be the poset of subsets of $\mathbb{N}\times \mathbb{N}$ ordered by
  inclusion regarded as a category, so that there is a morphism $A \rto B$ if
  $A \subseteq B$. Let $\W$ be the subcategory taking the morphisms in $\P$ for
  which the domain and the codomain differ by a finite number of elements.
  
  We claim that the pair $(\P,\W)$ satisfies all conditions of Theorem
  \ref{countableCharacterizationTheorem} except countability. First, since the weak
  classes are closed under products and coproducts, we deduce that in every weak class
  all elements are proto-centers. Second, if one weak class $A$ contains an element
  above an element of a weak class $B$, then every element of $A$ is above some element
  of $B$ and every element of $B$ is below some element of $A$. It follows that 
  every weak class contains a locally compatible proto-center.

  We claim that there is no model structure on $\P$ with $\W$ as category of
  weak equivalences.  By Theorem \ref{thm:modelexistsall}, it suffices to prove
  the nonexistence of a choice of center.

  Suppose for the sake of contradiction that a choice of centers $\chi$ for the
  pair $(\P,\W)$ exists.  Let $R_i=\{i\}\times\mathbb{N}$, and let $(i,k_i)$ be
  any element in $\chi(R_i) \cap R_i$.  Let
  \[X = \bigcup_{i\in \mathbb{N}} (\chi(R_i) \cap R_i - \{(i,k_i)\}).\] 
  There is a diagram
  \[R_i \stackrel{\sim}{\longleftarrow} \chi(R_i) \cap R_i -\{(i,k_i)\}
  \longrightarrow X\] for all $i$;  applying $\chi$ to this produces a morphism
  $f_i:\chi(R_i) \rto \chi(X)$.  Since the symmetric difference between $X$ and
  $\chi(X)$ is finite, there exists an $N$ such that for all $n \geq N$,
  \[\chi(X) \cap R_n = \chi(R_n) \cap R_n - \{(n,k_n)\}.\]
  Thus $\chi(R_n) \not\subseteq \chi(X)$ and $f_n$ cannot exist;
  contradiction.
\end{proof}

\section{Classification of model category structures on posets up to Quillen
  equivalence} \label{sec:comp}

We end this paper with an aside on uniqueness of model structures. We begin by
recalling the definition of Quillen equivalence:

\begin{definition}
  Given two categories $\C$ and $\D$ together with model structures $\CC$ and
  $\DD$, an adjoint pair of functors $F: \C \rightleftarrows \D\,:\! G$ is a
  \emph{Quillen adjunction} if $F$ preserves cofibrations and acyclic cofibrations and $G$ preserves
    fibrations and acyclic fibrations.  It is a \emph{Quillen equivalence} if moreover
  whenever $A\in \C$ is cofibrant and $B\in \D$ is fibrant then the
    morphism $A \rto G(B)$ is a weak equivalence if and only if its adjoint
    $F(A) \rto B$ is a weak equivalence.
  $\CC$ and $\DD$ are called \emph{Quillen equivalent} if there exists a chain of Quillen equivalences between them.
  For a model structure $\CC$, write $\Ho\CC := \C[\CC_{we}^{-1}]$.  If $\CC$
  and $\DD$ are Quillen equivalent then $\Ho\CC$ and $\Ho\DD$ are equivalent.
\end{definition}

We recall without proof some basic properties of Quillen equivalences.  For more details, see \cite[Section 16.2]{mayponto}.
\begin{lemma}
Let $F: \C \rightleftarrows \D \,:\!G$ be an adjoint pair of functors between model categories $\CC$ and $\DD$.
  \begin{enumerate}
  \item  $F$ preserves cofibrations and acyclic cofibrations if and only if $G$ preserves fibrations and acyclic fibration.
  \item If the adjuction is a Quillen adjunction, $F$ reflects weak equivalences and the counit of the adjuction is a weak equivalence for all fibrant objects then it is a Quillen equivalence.
  \end{enumerate}
\end{lemma}

Even if we know that a pair $(\C,\W)$ extends to a model structure there is
still the possibility for non-uniqueness: there might be two model structures
$\CC$ and $\CC'$ extending $(\C,\W)$ that are not Quillen equivalent.  Thus we
have the following question:
\begin{question} $\hbox{ }$ \label{q:MCequiv}
  \begin{enumerate}
  \item If $\CC$ and $\CC'$ are two model structures extending the pair
    $(\C,\W)$, are they Quillen equivalent?
  \item Moreover, if $\CC$ and $\CC'$ are Quillen equivalent, is it possible to construct a chain of Quillen equivalences in which every underlying functor is the identity functor?
  \end{enumerate}
\end{question}
We expect that the answer to (1) is ``yes'', even when $\C$ is not
a poset, and that the answer to (2) is ``yes'' when $\C$ is nice.
Intuitively, if we think of a model structure as a ``choice of coordinates'' on
a relative pair $(\C,\W)$, this says that all choices of coordinates are
equivalent.  

Although we cannot answer the question in general, in this section we prove that
when $\C$ is a poset the answer to (1) is ``yes,'' (Theorem~\ref{thm:C->HoC})
and when $\C$ is bicomplete and all weak equivalence classes in $\W$ are small
the answer to (2) is ``yes'' (Theorem~\ref{thm:modelsequiv}).

\begin{theorem} \label{thm:C->HoC} Let $\CC$ be a model structure on a preorder
  $\C$, and let $\D$ be the full subcategory of the cofibrant fibrant objects in
  $\C$.  Then $\CC$ is Quillen equivalent to the model structure $\DD$ on $\D$
  given by
  \[\DD_{we} = \iso\D \qquad \DD_{cof} = \DD_{fib} = \D.\]
\end{theorem}

In particular, this theorem shows that any two model structures on posets with
isomorphic homotopy categories are Quillen equivalent.  Embedded in the statement of
this theorem is the observation that $\Ho\CC$ must be finitely bicomplete.  In
fact, $\Ho\CC$ will have all limits and colimits that $\C$ does.

Most of the proof of this theorem is contained in the following proposition: 
\begin{proposition} \label{prop:cofibRepl} Let
  $\C^c$ be the full subcategory of cofibrant objects in $\C$.  We
  define \[\CC^c_{we} = \CC_{we}\cap \C^c \qquad \CC^c_{cof} = \CC_{cof}\cap
  \C^c \qquad \CC^c_{fib} = \CC_{fib} \cap \C^c.\] Then $\CC^c$ is model
  structure on $\C^c$ and the inclusion $\iota:\C^c \rto \C$ is the left adjoint in a Quillen
  equivalence $\CC^c \rightleftarrows \CC$.  
\end{proposition}

This proposition is a special case of \cite[Proposition 17]{balchingarner}; we
present the proof here as the poset case is easier to visualize than the
one-dimensional model structures in \cite{balchingarner}.

\begin{proof} First, note that $\C^c$ is bicomplete\footnote{We
    do not distinguish between finite and small in this case; $\C^c$ will have
    the same ones that $\C$ does.}.  It suffices to check that it has all
  products and coproducts, since equalizers and coequalizers are trivial in a
  poset.  An arbitrary coproduct of cofibrant objects is still cofibrant, so it
  suffices to check that $\C^c$ has all products.  Let $\{A_i\}_{i\in I}$ be a
  tuple of objects of $\C^c$, and let $B = \prod_{i\in I} A_i \in \C$.  We claim
  that the cofibrant replacement (unique by Lemma~\ref{uniqueness}) $B^c$ of $B$
  is the product of $A_i$ in $\C^c$.  Indeed, suppose that a cofibrant object
  $D$ has morphisms $D \rto A_i$ for all $i$.  Then we
  have a diagram \[\xymatrix{\initial \cofib[r] \cofib[d] & B^c \fib[d]^\sim \\
    D \ar[r] & B}\] which has a lift $h:D \rto B^c$.  Thus all cofibrant objects
  with morphisms to $B$ have morphisms to $B^c$.  As morphisms are uniquely
  determined by their source and target this makes $B^c$ into the product of the
  $A_i$ inside $\C^c$, as desired.  Factorizations in $\C$ yield factorizations
  in $\C^c$, so by \cite[14.1.13]{mayponto} $\CC^c$ is a model structure.  We
  define a right adjoint $\gamma$ to $\iota$ by sending each object $A$ to its
  cofibrant replacement; by Lemma \ref{uniqueness}, this is well-defined.  By Lemma \ref{uniqueness} again, $\gamma(\iota(A)) = A$, so
  the unit of the adjunction is the identity transformation.  The counit of the
  transformation is the acyclic fibration $\gamma(A) \acycfib A$.  As $\iota$
  preserves cofibrations and weak equivalences by definition, it is the left adjoint in a
  Quillen adjunction.  Since $\iota$ reflects weak equivalences and the counit of the adjunction is a natural weak
  equivalence, the adjuction is a Quillen equivalence, as desired.
\end{proof} 

We can now prove Theorem~\ref{thm:C->HoC}.  
\begin{proof}[Proof of Theorem \ref{thm:C->HoC}] 
  Let $\C^c$ be the full subcategory of $\C$ containing all cofibrant objects
  and $\CC^c$ be the model structures defined on it by
  Proposition~\ref{prop:cofibRepl}.  By Proposition \ref{prop:cofibRepl} $\CC$
  is Quillen equivalent to $\CC^c$.  Note that $\DD_{cof} = \CC^c_{cof}
  \cap \D$ and $\DD_{fib} = \CC^c_{fib} \cap \D$.  By the dual of Proposition
  \ref{prop:cofibRepl}, $\CC^c$ is Quillen equivalent to $\DD$.
\end{proof}

As an amusing aside, this allows us to classify which model structures on posets
are cofibrantly generated:

\begin{theorem} Let $\C$ be any bicomplete poset, and let $\CC$ a model
  structure on $\C$.  If $\C$ is small then $\CC$ is cofibrantly generated;
  conversely, if $\CC$ is cofibrantly generated then $\C$ is right Quillen
  equivalent to a small model category.
\end{theorem} 

\begin{proof} If $\C$ is small then $\CC$ is trivially cofibrantly generated: we
  define the set of generating cofibrations to be the set of all
  cofibrations, and the set of generating acyclic cofibrations to be the set of
  all acyclic cofibrations.

  Now suppose that $\CC$ is cofibrantly generated.  By Proposition
  \ref{prop:cofibRepl} it suffices to show that $\C^c$ is small.  Let $S =
  \{f_i:A_i \rto B_i\}$ be the set of generating cofibrations.  By the small
  object argument, for any object $X\in \C$ we construct its cofibrant
  replacement $\gamma(X)$ by defining $\gamma_0(X) = \emptyset$ and setting
  $\gamma_{n+1}(X)$ to be the pushout of
  \[\xymatrix{\displaystyle{\coprod_{f_i\in S_n} B_i} \ar@{<-^)}[r] & \displaystyle{\coprod_{f_i\in S_n} A_i}  \ar[r] &
    \gamma_n(X)}\]
  where 
  \[S_n = \{f_i\in S\,|\, \Hom(A_i,\gamma_n(X))\times \Hom(B_i, X) \neq
  \emptyset\}.\] Then there is a cofibration $\gamma_n(X) \rto \gamma_{n+1}(X)$
  and the cofibrant replacement of $X$ is $\colim_n \gamma_n(X)$.  Note that by
  definition, $S_n \subseteq S_{n+1}$ for all $n$.  Observe that for any
  nonempty set $T$ and any object $A\in \C$ we have $\coprod_T A = A$; therefore
  \[\gamma_{n+1}(X) = \gamma_n(X) \amalg \coprod_{f_i\in S_n} B_i \cong
  \coprod_{f_i\in S_n} B_i.\]
  Thus if we set $S_\infty =
  \bigcup_{n\geq 0} S_n$ it follows that $\gamma(X) = \coprod_{f_i\in S_\infty} B_i$.
  In particular, all cofibrant replacements correspond to subsets of $S$; as $S$
  is a set, the class of cofibrant objects must also be a set.  Therefore $\C^c$
  is small.
\end{proof}

This theorem allows us to construct model categories which are neither
cofibrantly nor fibrantly generated.
\begin{corollary} \label{cor:non-cg}
  Let $\CC$ be a model structure on a bicomplete poset $\C$ such that $\C^c$ has size
  $2^\kappa$ for some cardinal $\kappa$.  Then if $\CC$ is cofibrantly generated
  it must have at least $\kappa$ generators.  In particular, if $\C^c$ is not
  small then $\CC$ is not cofibrantly generated.  Dually, if $\C^f$ is not small
  then $\CC$ is not fibrantly generated.
\end{corollary}

\begin{ex}
  Corollary~\ref{cor:non-cg} and Theorem~\ref{thm:modelexistsall} give an
  interesting method for producing non-cofibrantly or fibrantly generated
  model structures on posets.  Let $\C$ be any large bicomplete poset.  Take any
  collection $\{A_i\}_{i\in \ob\C}$, where each $A_i$ is a bicomplete poset with
  terminal object $*_i$.  Write
  \[i: \coprod_{i\in \ob\C} \ob A_i \rto \ob \C\]
  for the map taking $X\in A_i$ to $i$.  Let $\tilde \C$ be the poset whose
  objects are $\coprod_{i\in \ob \C} \ob A_i$ and where
  \[\Hom(X,Y) =
    \begin{cases}
      \Hom_{A_i}(X,Y) & \hbox{if } i(X) = i(Y) \\
      \Hom_\C(i(X),i(Y)) &\hbox{if }  i(X) \neq i(Y).
    \end{cases}\] Then $i$ extends to a functor $i: \tilde \C \rto \C$; we
  define $\W$ to be the preimage of the identity morphisms.

  To check that $\C$ is bicomplete it suffices to check that all products and
  coproducts exist. We check products, as coproducts follow analogously.  Let
  $\{B_\alpha\}_{\alpha\in A}$ be a collection of objects in $\tilde \C$, Let
  $i:\tilde \C \rto \C$ be the functor taking an object to its indexing element
  in $\C$. Let $i' = \prod_{\alpha\in A} i(A_\alpha)$, and let $*_{i'}$ be the
  terminal object in $A_{i'}$.  Then
  \[\prod_{\alpha\in A} B_\alpha = \prod_{\substack{\alpha\in A \\ i(\alpha) = i'}} B_\alpha \times *_{i'}.\]

  The subcategory $\W$ is decomposable.  To see this, note that if $X \rto Y$ is
  a morphism in $\W$ then we must have $i(X) = i(Y)$, and thus this is a
  morphism in $A_i$.  If a composition $X \rto Y \rto Z$ is in $\W$ then the
  composition $i(X) \rto i(Y) \rto i(Z)$ is an identity; in particular, since
  $\C$ is a poset we must have $i(Y) = i(X)$, and thus $X \rto Y$ and $Y \rto Z$
  are also in $\W$.   
  
  The map $X \mapsto *_{i(X)}$ is a choice of centers, and by
  Theorem~\ref{thm:modelexistsall} this pair extends to a model structure.
  However, since $\C$ is large this model is neither cofibrantly nor fibrantly
  generated.
\end{ex}

We turn to the second half of Question~\ref{q:MCequiv}. In the case of
posets where every weak equivalence class of $\W$ is small and $\C$ is
bicomplete we answer it.

\begin{theorem} \label{thm:modelsequiv}
  If all weak equivalence classes of $\W$ are small and $\C$ is bicomplete then
  any two model structures extending $(\C,\W)$ are Quillen equivalent via a
  zigzag of equivalences each of whose underlying functors is the identity.
\end{theorem}

We prove this theorem by constructing a ``minimal'' model structure in which a
fixed class $J$ of weak equivalences are actually acyclic cofibrations.  The
construction of this structure is where the assumptions on $\W$ and $\C$ are
necessary.  We begin with a couple of extra technical results about lifting
systems. 

\begin{definition}
  A poset $\C$ is \emph{left-small} with respect to class $\LL$ if for all
  objects $A \in \C$, the class $\{f\in \LL\,|\, \mathrm{dom}\, f = A\}$ is a
  set.  Dually, $\C$ is \emph{right-small} with respect to $\RR$ if for all
  objects $A\in \C$ the class $\{f\in\RR\,|\, \mathrm{codom}\,f = A\}$ is a set.
\end{definition}

\begin{lemma} \label{lem:lifts} Suppose that $(\LL,\RR)$ is a pair of classes of
  morphisms such that $\LL^\llp = \RR$ and $\LL$ contains all isomorphisms.  If
  $\C$ is left-small with respect to $\LL$ and $\LL$ is closed under
  compositions, pushouts in $\C$ and arbitrary coproducts then $(\LL, \RR)$ is a
  WFS.
\end{lemma}

\begin{proof}
  By \cite[14.1.13]{mayponto}, in order to show that $(\LL, \RR)$ is a WFS it
  suffices to show that every morphism $f:A \rto B$ in $\C$ can be factored as
  $f_Rf_L$, where $f_L\in \LL$ and $f_R\in \RR$.

  Let $f:A \rto B$ be any morphism in $\C$.  Let 
  \[S = \{A'\in \C \,|\, A \rto A' \rto B,\ A \rto A'\in \LL\}\] and let $\tilde
  A = \colim S$; note that this is well-defined since $\C$ is left-small with
  respect to $\LL$, $S$ is a set and $\C$ contains all colimits.  The morphism
  $A \rto \tilde A$ can be written as $\coprod_{A'\in S} (A \rto A')$, so $A
  \rto \tilde A\in \LL$.  We claim that $A \rto \tilde A \rto B$ gives the
  desired factorization.  By the definition of $\tilde A$ there are no
  factorizations of $\tilde A \rto B$ through non-invertible morphisms $\tilde A
  \rto Z\in \LL$; thus by Lemma~\ref{lem:see-lift} $\tilde A \rto B\in
  \LL^\llp$, as desired.
\end{proof}

\begin{corollary} \label{cor:max-wfs} If $(\LL, \RR)$ is a MLS (see Definition~\ref{def:mls}) and $\C$ is
  left-small with respect to $\LL$ or right-small with respect to $\RR$ then
  $(\LL,\RR)$ is a WFS.
\end{corollary}

We are now ready to construct our minimal model structure.

\begin{proposition} \label{prop:genMC} Let $\C$ be a bicomplete poset, $\W$ a
  decomposable subcategory and $J$ a class of morphisms in $\W$.  We define
  \[\CC^J_{we} = \W \qquad \CC^J_{fib} = J^\llp \qquad \CC^J_{cof} =
  {}^\llp(\CC^J_{we}\cap \CC^J_{fib}).\]
  $\CC^J$ is a model structure if the following extra assumptions hold:
  \begin{enumerate}
  \item \label{genMCstr} All connected components of $\W$ are small.
  \item \label{genMCass2} $(\CC^J_{cof})^\llp \subseteq \W$.
  \item \label{genMCass3} ${}^\llp \CC^J_{fib} \subseteq \W$.
  \end{enumerate}
\end{proposition}

This proposition also has a dual version, which defines a model structure $^J\CC$ with $^J\CC_{cof} = {}^\llp J$.

\begin{proof}
  Since $\W$ is decomposable it must also satisfy (2OF3).  By (\ref{genMCstr})
  $\C$ is left-small with respect to $\CC^J_{cof}\cap \CC^J_{we}$ and right-small
  with respect to $\CC^J_{fib}\cap \CC^J_{we}$.  Thus by Corollary~\ref{cor:max-wfs}
  in order to show that (WFS) holds it suffices to check that both $(\CC^J_{cof},
  \CC^J_{fib}\cap \CC^J_{we})$ and $(\CC^J_{cof}\cap \CC^J_{we}, \CC^J_{fib})$ are MLSs.

  To check that a pair $(\LL,\RR)$ is an MLS we must check that  $\LL = {}^\llp \RR$ and that $\LL^\llp \subseteq \RR$.  We begin by showing that the first of these holds for both pairs. For $(\CC^J_{cof},
  \CC^J_{fib}\cap \CC^J_{we})$, this is true by definition. Now consider  $(\CC^J_{cof}\cap \CC^J_{we}, \CC^J_{fib})$. Suppose that $i:A \rto
  B$ is in $\CC^J_{cof}\cap \CC^J_{we}$.  By the dual of Lemma~\ref{lem:see-lift},
  $i\llp \CC^J_{fib}$ if and only if all factorizations $A \rto Z \rto B$ with $Z
  \rto B$ in $\CC^J_{fib}$ have $Z = B$.  As $\W$ is decomposable such a factorization has $Z
  \rto B$ in $\CC^J_{we}$; since $i\in \CC^J_{cof}$ it lifts on the left of $Z\rto
  B$ and we must have $Z = B$, as desired.  Thus $\CC^J_{cof} \cap \CC^J_{we} \subseteq {}^\llp\CC_{fib}$.  On the other hand,  by assumption (\ref{genMCass3}),  $^\llp\CC_{fib} \subseteq \W$ and
  \[^\llp\CC^J_{fib} \subseteq {}^\llp (\CC^J_{fib} \cap \CC^J_{we}) = \CC^J_{cof}.\]
  Thus $\CC^J_{cof}\cap \CC^J_{we} \supseteq {}^\llp\CC^J_{fib}$ and equality holds.
  
   We now turn to proving that $\LL^\llp \subseteq \RR$ for both pairs.

  First consider $(\CC^J_{cof}, \CC^J_{fib}\cap \CC^J_{we})$. Note that
  \[J \subseteq {}^\llp(J^\llp) = {}^\llp \CC^J_{fib} \subseteq {}^\llp(\CC^J_{fib}
  \cap \CC^J_{we}) = \CC^J_{cof}.\]
  Thus $(\CC_{cof}^J)^\llp \subseteq J^\llp = \CC^J_{fib}$.  By assumption (\ref{genMCass2}) we know
  that it is also a subset of $\CC^J_{we}$, as desired.
  Now consider $(\CC^J_{cof}\cap \CC^J_{we}, \CC^J_{fib})$. From before, $J\subseteq \CC^J_{cof}\cap \CC^J_{we}$, from which it
  follows that
  \[(\CC^J_{cof}\cap \CC^J_{we})^\llp \subseteq J^\llp = \CC^J_{fib}.\]
\end{proof}

We can now start comparing different model structures on a poset.  

\begin{proposition} \label{prop:newcofib} Suppose that $\CC$ is any model
  structure on $\C$ with $\CC_{we} = \W$ and $\chi$ is any choice of centers for $\W$.  Let $J = (\CC_{cof} \cap \CC_{we})
  \cup J_\chi$.  
  If all connected components in $\W$ are small,
  then $\CC^J$ is another model structure on $\C$, and the identity functor gives
  a Quillen equivalence $\CC\rightleftarrows \CC^J$.  
\end{proposition}

\begin{proof}
  We use
  Proposition \ref{prop:genMC}.  Condition (\ref{genMCstr}) is assumed, so we
  just need to check (\ref{genMCass2}) and (\ref{genMCass3}).

  \noindent
  (\ref{genMCass2}) Note that $\CC^J_{fib} \subseteq \CC_{fib}$.  Thus $\CC^J_{cof} \supseteq \CC_{cof}$ and $(\CC^J_{cof})^\llp \subseteq \CC_{cof}^\llp \subseteq \W$.

  \noindent
  (\ref{genMCass3}) Suppose $f:A \rto B$ is such that $1_A\in Q_\chi$ and $f\llp
  \CC'_{fib}$.  Factor $f = f_ff_{ac}$ with $f_f:B' \rto B\in\CC_{fib}$ and
  $f_{ac}\in \CC_{cof}\cap \CC_{we}$.  We claim that $f_f\in \CC^J_{fib}$.  Because $\CC_{cof} \cap \CC_{we} \llp f_f$, it suffices to
  show that $\W_c^\chi \llp f_f$.  By Lemma~\ref{lem:see-lift}, it suffices to
  check that any factorization of $f_f$ as $B' \rto Z \rto B$ with $B' \rto Z$
  in $\W_c^\chi$ must have $B' = Z$.  But $B' \rto Z$ is in $Q_\chi$, so by definition it is an identity and $B' = Z$.
  Since $f_f\in \CC^J_{fib}$, $f\llp f_f$ and we must have $f = f_{ac} \in
  \CC'_{we}$.

  Now suppose $f\in {}^\llp\CC^J_{fib}$ is arbitrary.  Let $f'$ be the pushout of
  $f$ along $A \rto A\cup \chi(A)$.  Since $f'$ is the pushout of $f$ it is also
  in ${}^\llp\CC'_{fib}$, and as $A \rto A\cup\chi(A)$ is in $\CC^J_{we}$, $f$ is
  a weak equivalence if and only if $f'$ is.  By the above, $f'\in \CC^J_{we}$,
  so we conclude that so is $f$.

  The identity functor gives a Quillen equivalence $\CC\rightleftarrows \CC^J$ because
  the weak equivalences of the two structures are the same, and $\CC_{cof}
  \subseteq \CC^J_{cof}$.
\end{proof}

We are now ready to prove Theorem~\ref{thm:modelsequiv}.

\begin{proof}[Proof of Theorem \ref{thm:modelsequiv}] 
  Let $\CC_1,\CC_2$ be any model structures on $\C$ with weak equivalences $\W$,
  and let $\chi_i$ be the choice of centers given by $\CC_i$.  Let $J_i = ((\CC_i)_{cof}\cap\W) \cup J_{\chi_i}$ and let $\CC^{\chi_i}$ be the model structure constructed in
  Definition~\ref{def:Cchi}.  Let $\chi = \chi_1\times \chi_2$; by
  Lemma~\ref{lem:center-product} this is another choice of centers.  Note that
  $\W_c^{\chi_i} \subseteq \W_c^{\chi}$ since $Q_{\chi} \subseteq Q_{\chi_i}$.
  Thus $(\W_c^{\chi_i})^\llp \supseteq (\W_c^{\chi})^\llp$, and the identity
  functor gives a Quillen equivalence $\CC^{\chi_i} \rightleftarrows \CC^{\chi}$.  Thus for $i = 1,2$ the identity functor gives a zigzag of Quillen equivalences 
  \[\CC_i \rightleftarrows \CC^{J_i} \leftrightarrows \CC^{\chi_i} \rightleftarrows
  \CC^\chi,\] and the two model structures are equivalent.
\end{proof}


\begin{thebibliography}{Dug01}

\bibitem[Ber07]{bergner07}
Julia~E. Bergner.
\newblock Three models for the homotopy theory of homotopy theories.
\newblock {\em Topology}, 46(4):397--436, 2007.

\bibitem[Cis06]{cisinski06}
Denis-Charles Cisinski.
\newblock Les pr\'efaisceaux comme mod\`eles des types d'homotopie.
\newblock {\em Ast\'erisque}, 308:xxiv+390, 2006.

\bibitem[Dug01]{dugger01}
Daniel Dugger.
\newblock Replacing model categories with simplicial ones.
\newblock {\em Trans. Amer. Math. Soc.}, 353(12):5003--5027, 2001.

\bibitem[Hir03]{hirschhorn}
Philip~S. Hirschhorn.
\newblock {\em Model categories and their localizations}, volume~99 of {\em
  Mathematical Surveys and Monographs}.
\newblock American Mathematical Society, Providence, RI, 2003.

\bibitem[Hov99]{hovey99}
Mark Hovey.
\newblock {\em Model categories}, volume~63 of {\em Mathematical Surveys and
  Monographs}.
\newblock American Mathematical Society, Providence, RI, 1999.

\bibitem[Mal36]{maltsev}
A.~I. Maltsev.
\newblock Untersuchungen aus dem gebiete der mathematischen logik.
\newblock {\em Matematicheskii Sbornik}, 1:323--336, 1936.

\bibitem[MP12]{mayponto}
J.~P. May and K.~Ponto.
\newblock {\em More concise algebraic topology}.
\newblock Chicago Lectures in Mathematics. University of Chicago Press,
  Chicago, IL, 2012.
\newblock Localization, completion, and model categories.

\bibitem[Qui67]{quillen67}
Daniel~G. Quillen.
\newblock {\em Homotopical algebra}.
\newblock Lecture Notes in Mathematics, No. 43. Springer-Verlag, Berlin-New
  York, 1967.

\bibitem[Rie14]{riehl14}
Emily Riehl.
\newblock {\em Categorical homotopy theory}, volume~24 of {\em New Mathematical
  Monographs}.
\newblock Cambridge University Press, Cambridge, 2014.

\end{thebibliography}


\begin{thebibliography}{DHKS04}

\bibitem[Ber07]{bergner07}
Julia~E. Bergner.
\newblock Three models for the homotopy theory of homotopy theories.
\newblock {\em Topology}, 46(4):397--436, 2007.

\bibitem[BG19]{balchingarner}
Scott Balchin and Richard Garner.
\newblock Bousfield localisation and colocalisation of one-dimensional model
  structures.
\newblock {\em Appl. Categ. Structures}, 27(1):1--21, 2019.

\bibitem[BK12]{barwickkan}
C.~Barwick and D.~M. Kan.
\newblock Relative categories: another model for the homotopy theory of
  homotopy theories.
\newblock {\em Indag. Math. (N.S.)}, 23(1-2):42--68, 2012.

\bibitem[Cis06]{cisinski06}
Denis-Charles Cisinski.
\newblock Les pr\'efaisceaux comme mod\`eles des types d'homotopie.
\newblock {\em Ast\'erisque}, 308:xxiv+390, 2006.

\bibitem[DHKS04]{DHKS}
William~G. Dwyer, Philip~S. Hirschhorn, Daniel~M. Kan, and Jeffrey~H. Smith.
\newblock {\em Homotopy limit functors on model categories and homotopical
  categories}, volume 113 of {\em Mathematical Surveys and Monographs}.
\newblock American Mathematical Society, Providence, RI, 2004.

\bibitem[DK80]{dwyerkan1}
W.~G. Dwyer and D.~M. Kan.
\newblock Simplicial localizations of categories.
\newblock {\em J. Pure Appl. Algebra}, 17(3):267--284, 1980.

\bibitem[DK87]{dwyerkan2}
W.~G. Dwyer and D.~M. Kan.
\newblock Equivalences between homotopy theories of diagrams.
\newblock In {\em Algebraic topology and algebraic {$K$}-theory ({P}rinceton,
  {N}.{J}., 1983)}, volume 113 of {\em Ann. of Math. Stud.}, pages 180--205.
  Princeton Univ. Press, Princeton, NJ, 1987.

\bibitem[Dug01]{dugger01}
Daniel Dugger.
\newblock Replacing model categories with simplicial ones.
\newblock {\em Trans. Amer. Math. Soc.}, 353(12):5003--5027, 2001.

\bibitem[Hir03]{hirschhorn}
Philip~S. Hirschhorn.
\newblock {\em Model categories and their localizations}, volume~99 of {\em
  Mathematical Surveys and Monographs}.
\newblock American Mathematical Society, Providence, RI, 2003.

\bibitem[Hov99]{hovey99}
Mark Hovey.
\newblock {\em Model categories}, volume~63 of {\em Mathematical Surveys and
  Monographs}.
\newblock American Mathematical Society, Providence, RI, 1999.

\bibitem[Joy02]{joyal}
A.~Joyal.
\newblock Quasi-categories and {K}an complexes.
\newblock {\em J. Pure Appl. Algebra}, 175(1-3):207--222, 2002.
\newblock Special volume celebrating the 70th birthday of Professor Max Kelly.

\bibitem[Lur09]{lurie}
Jacob Lurie.
\newblock {\em Higher topos theory}, volume 170 of {\em Annals of Mathematics
  Studies}.
\newblock Princeton University Press, Princeton, NJ, 2009.

\bibitem[Mal36]{maltsev}
A.~I. Maltsev.
\newblock Untersuchungen aus dem gebiete der mathematischen logik.
\newblock {\em Matematicheskii Sbornik}, 1:323--336, 1936.

\bibitem[MP12]{mayponto}
J.~P. May and K.~Ponto.
\newblock {\em More concise algebraic topology}.
\newblock Chicago Lectures in Mathematics. University of Chicago Press,
  Chicago, IL, 2012.
\newblock Localization, completion, and model categories.

\bibitem[Qui67]{quillen67}
Daniel~G. Quillen.
\newblock {\em Homotopical algebra}.
\newblock Lecture Notes in Mathematics, No. 43. Springer-Verlag, Berlin-New
  York, 1967.

\bibitem[Rie14]{riehl14}
Emily Riehl.
\newblock {\em Categorical homotopy theory}, volume~24 of {\em New Mathematical
  Monographs}.
\newblock Cambridge University Press, Cambridge, 2014.

\bibitem[RT07]{rosickytholen}
Ji\v{r}\'{\i} Rosick\'{y} and Walter Tholen.
\newblock Factorization, fibration and torsion.
\newblock {\em J. Homotopy Relat. Struct.}, 2(2):295--314, 2007.

\bibitem[Str72]{strom}
Arne Str{\o}m.
\newblock The homotopy category is a homotopy category.
\newblock {\em Arch. Math. (Basel)}, 23:435--441, 1972.

\bibitem[Tho80]{thomason}
R.~W. Thomason.
\newblock Cat as a closed model category.
\newblock {\em Cahiers Topologie G\'eom. Diff\'erentielle}, 21(3):305--324,
  1980.

\end{thebibliography}
\end{document}